\documentclass[reqno,11pt]{article} 
\usepackage[left=2.5cm,right=2.5cm,top=3cm,bottom=3cm,a4paper]{geometry}
\usepackage{amsmath, amssymb}
\usepackage{amsthm, amscd} 
\usepackage[all,cmtip]{xy}
\usepackage{float}
\usepackage{caption}
\usepackage{color}

\newcommand{\Mod}[1]{\ (\textup{mod}\ #1)}
\theoremstyle{plain} 
\newtheorem{theorem}{\indent\sc Theorem}[section]
\newtheorem{lemma}[theorem]{\indent\sc Lemma}
\newtheorem{corollary}[theorem]{\indent\sc Corollary}
\newtheorem{proposition}[theorem]{\indent\sc Proposition}

\theoremstyle{definition} 

\newtheorem{remark}[theorem]{\indent\sc Remark}
\newtheorem{example}[theorem]{\indent\sc Example}

\makeatletter
\def\address#1#2{\begingroup
\noindent\parbox[t]{7.8cm}{%
\small{\scshape\ignorespaces#1}\par\vskip1ex
\noindent\small{\itshape E-mail address}%
\/: #2\par\vskip4ex}\hfill%
\endgroup}%
\makeatother
\title{On some extension of Gauss' work and applications (II)}
\author{
\textsc{Ick Sun Eum, Ho Yun Jung, Ja Kyung Koo and Dong Hwa Shin} 
}
\date{} 
\begin{document}

\allowdisplaybreaks

\maketitle

 \footnote{
\thanks{
The first author was supported by the Dongguk University Research Fund of 2018 and the National Research Foundation of Korea (NRF) grant funded by the Korea government (MSIT) (2017R1C1B5017567).
The second (corresponding) author was supported by the National Research Foundation of Korea (NRF) grant funded by the Korea Government (MIST) (2016R1A5A1008055 and
2017R1C1B2010652).
The fourth author was supported by the Hankuk University of Foreign Studies Research Fund of 2018 and the National Research Foundation of Korea (NRF) grant funded by the Korea government (MSIT) (2017R1A2B1006578).}
}

\begin{abstract}
Let $K$ be an imaginary quadratic field of discriminant $d_K$, and let $\mathfrak{n}$ be a
nontrivial integral ideal of $K$ in which $N$ is the smallest positive integer. Let
$\mathcal{Q}_N(d_K)$ be the set
of primitive positive definite binary quadratic forms of discriminant
$d_K$ whose leading coefficients are relatively prime to $N$. We adopt an equivalence relation
$\sim_\mathfrak{n}$ on $\mathcal{Q}_N(d_K)$ so that the set of equivalence classes $\mathcal{Q}_N(d_K)/\sim_\mathfrak{n}$
can be regarded as a group isomorphic to the ray class group of $K$ modulo $\mathfrak{n}$.
We further present an explicit isomorphism
of $\mathcal{Q}_N(d_K)/\sim_\mathfrak{n}$ onto $\mathrm{Gal}(K_\mathfrak{n}/K)$ in terms of Fricke invariants,
where $K_\mathfrak{n}$ is the ray class field of $K$ modulo $\mathfrak{n}$. This would be certain extension of the classical composition theory of  binary quadratic forms, originated and developed by Gauss and Dirichlet.
\end{abstract}

\maketitle

\section {Introduction}

For a negative integer $D$ such that $D\equiv0$ or $1\Mod{4}$, let $\mathcal{Q}(D)$ be the set of
primitive positive definite binary quadratic forms
$Q(x,\,y)=ax^2+bxy+cy^2$ of discriminant $b^2-4ac=D$.
Then, the modular group $\mathrm{SL}_2(\mathbb{Z})$ (or
$\mathrm{PSL}_2(\mathbb{Z})$) naturally acts on $\mathcal{Q}(D)$
and gives rise to
the proper equivalence $\sim$  as
\begin{equation*}
Q\sim Q'\quad\Longleftrightarrow\quad
Q=Q'^{\,\gamma}=
Q'\left(\gamma\begin{bmatrix}x\\y\end{bmatrix}\right)~
\textrm{for some}~\gamma\in\mathrm{SL}_2(\mathbb{Z}).
\end{equation*}
It was Gauss who first established a
systematic theory of composition for binary quadratic forms
in his monumental work \textit{Disquisitiones
arithmeticae} published in 1801 (\cite{Gauss}).
More precisely, the set of equivalence classes $\mathcal{C}(D)=\mathcal{Q}(D)/\sim$ becomes an abelian group under the direct composition, called the
\textit{form class group} of discriminant $D$.
Besides, in 2004 some elegant and enlightening progress was achieved
in higher composition theory by Bhargava (\cite{BhargavaI}, \cite{BhargavaII}
and \cite{BhargavaIII}).
See also \cite{Buell12}.
\par
Let $K=\mathbb{Q}(\sqrt{D})$ and $\mathcal{O}$ be the order
of discriminant $D$ in the imaginary quadratic field $K$.
Let $I(\mathcal{O})$ be the group of
proper fractional $\mathcal{O}$-ideals, and
let $P(\mathcal{O})$ be its subgroup of principal $\mathcal{O}$-ideals.
For each $Q=ax^2+bxy+cy^2\in\mathcal{Q}(D)$, let
$\omega_Q$ be the zero of $Q(x,\,1)$ in the complex
upper half-plane $\mathbb{H}$, namely,
\begin{equation*}
\omega_Q=\frac{-b+\sqrt{D}}{2a}.
\end{equation*}
Then, it is well known that
the form class group $\mathcal{C}(D)$ is
isomorphic to the $\mathcal{O}$-ideal class group
$\mathcal{C}(\mathcal{O})=
I(\mathcal{O})/P(\mathcal{O})$ through the isomorphism
\begin{equation}\label{CC}
\mathcal{C}(D)\stackrel{\sim}{\rightarrow}\mathcal{C}(\mathcal{O}),
\quad [Q]\mapsto[[\omega_Q,\,1]]=[\mathbb{Z}\omega_Q+\mathbb{Z}].
\end{equation}
As noted in \cite[Theorem 7.7]{Cox} it is not necessary to prove
that $\mathcal{C}(D)$ is a group. It suffices to show that the mapping stated in (\ref{CC})
is a well-defined bijection, through which the $\mathcal{O}$-ideal class group $\mathcal{C}(\mathcal{O})$ equips $\mathcal{C}(D)$
with a group structure, namely, the Dirichlet composition (\cite[(3.7)]{Cox} and \cite{Dirichlet}).
\par
On the other hand, if $H_\mathcal{O}$ denotes the ring class field of order $\mathcal{O}$, then it is generated by the $j$-invariant $j(\mathcal{O})$ by the theory of complex multiplication.
Furthermore, there is an isomorphism
\begin{equation}\label{CG}
\mathcal{C}(\mathcal{O})\stackrel{\sim}{\rightarrow}\mathrm{Gal}(H_\mathcal{O}/K),\quad
[\mathfrak{a}]\mapsto\left(j(\mathcal{O})\mapsto
j(\overline{\mathfrak{a}})\right)
\end{equation}
(\cite[Theorem 11.1 and Corollary 11.37]{Cox}). Thus we attain the isomorphism
\begin{equation*}
\mathcal{C}(D)\stackrel{\sim}{\rightarrow}\mathrm{Gal}(H_\mathcal{O}/K),\quad
[Q]\mapsto\left(j(\mathcal{O})\mapsto j([-\overline{\omega}_Q,\,1])\right)
\end{equation*}
by (\ref{CC}) and (\ref{CG}).
\par
For a nontrivial ideal $\mathfrak{n}$ of
the maximal order
$\mathcal{O}_K$ of $K$, let $\mathcal{C}(\mathfrak{n})$ be
the ray class group modulo $\mathfrak{n}$, namely,
\begin{equation*}
\mathcal{C}(\mathfrak{n})=I_K(\mathfrak{n})/P_{K,\,1}(\mathfrak{n})
\end{equation*}
where $I_K(\mathfrak{n})$ is the group of fractional
ideals of $K$ relatively prime to $\mathfrak{n}$
and $P_{K,\,1}(\mathfrak{n})$ is its subgroup given by
\begin{equation*}
P_{K,\,1}(\mathfrak{n})=\{\nu\mathcal{O}_K~|~
\nu\in K^*~\textrm{such that}~\nu\equiv^*1\Mod{\mathfrak{n}}\}.
\end{equation*}
The \textit{ray class field} modulo $\mathfrak{n}$,
denoted by $K_\mathfrak{n}$, is defined
to be a unique abelian extension of $K$
with Galois group isomorphic to $\mathcal{C}(\mathfrak{n})$
in which every ramified prime ideal divides $\mathfrak{n}$.
One can refer to \cite[Chapter V]{Janusz} or \cite{Neukirch} for the class field theory.
\par
For a positive integer $N$, let
\begin{equation*}
\mathcal{Q}_N(d_K)=\{
ax^2+bxy+cy^2\in\mathcal{Q}(d_K)~|~\gcd(N,\,a)=1\},
\end{equation*}
and let $\Gamma_1(N)$ be
the congruence subgroup of $\mathrm{SL}_2(\mathbb{Z})$ defined by
\begin{equation*}
\Gamma_1(N)=\left\{\gamma\in\mathrm{SL}_2(\mathbb{Z})~|~
\gamma\equiv\begin{bmatrix}1&\mathrm{*}\\
0&1\end{bmatrix}\Mod{NM_2(\mathbb{Z})}\right\}.
\end{equation*}
Here, $M_2(\mathbb{Z})$ means the $\mathbb{Z}$-module of
$2\times2$ matrices over $\mathbb{Z}$.
One can define an equivalence relation $\sim_{\Gamma_1(N)}$
on $\mathcal{Q}_N(d_K)$ by
\begin{equation}\label{QQQQ}
Q\sim_{\Gamma_1(N)}Q'\quad
\Longleftrightarrow\quad
Q=Q'^{\,\gamma}~\textrm{for some}~\gamma\in\Gamma_1(N).
\end{equation}
Recently,
Jung, Koo and Shin showed that if
$\mathfrak{n}=N\mathcal{O}_K$, then
the mapping
\begin{equation}\label{QGC}
\phi_N:\mathcal{Q}_N(d_K)/\sim_{\Gamma_1(N)}\rightarrow
\mathcal{C}(\mathfrak{n}),\quad
[Q]\mapsto[[\omega_Q,\,1]]
\end{equation}
is a well-defined bijection (\cite[Theorem 2.5 and Proposition 5.3]{J-K-S18}), through which
$\mathcal{Q}_N(d_K)/\sim_{\Gamma_1(N)}$
can be regarded as a group isomorphic to
$\mathcal{C}(\mathfrak{n})$.
Let $\mathcal{F}_{1,\,N}(\mathbb{Q})$ be
the field of meromorphic modular functions for
$\Gamma_1(N)$ with rational Fourier coefficients.
By using the Shimura reciprocity law, they further provided the isomorphism
\begin{equation*}
\mathcal{Q}_N(d_K)/\sim_{\Gamma_1(N)}
\stackrel{\sim}{\rightarrow}
\mathrm{Gal}(K_\mathfrak{n}/K),
\quad[Q]\mapsto\left(h(\tau_K)\mapsto
h(-\overline{\omega}_Q)~|~h(\tau)~\in
\mathcal{F}_{1,\,N}(\mathbb{Q})~
\textrm{is finite at $\tau_K$}\right)
\end{equation*}
where $\tau_K$ is the CM-point associated with the principal form of discriminant $d_K$.
\par
In this paper, for an arbitrary nontrivial ideal
$\mathfrak{n}$ of $\mathcal{O}_K$ where $N$ is the
least positive integer,
we shall define an equivalence relation $\sim_\mathfrak{n}$
on $\mathcal{Q}_N(d_K)$ so that the canonical mapping
\begin{eqnarray}\label{canonical}
\phi_\mathfrak{n}~:~
\mathcal{Q}_N(d_K)/\sim_\mathfrak{n}&\rightarrow&
\mathcal{C}(\mathfrak{n})\\
\mathrm{[}Q\mathrm{]}~~&\mapsto&[[\omega_Q,\,1]]\nonumber
\end{eqnarray}
becomes a well-defined bijection (Theorem \ref{canonicalisomorphism}).
In fact, the equivalence relation $\sim_\mathfrak{n}$ turns
out to be induced from a congruence subgroup of level $N$
(Proposition \ref{inducedfrom}).
Thus we may consider $\mathcal{Q}_N(d_K)/\sim_\mathfrak{n}$ as a group isomorphic to $\mathcal{C}(\mathfrak{n})$
whose group operation
will be described in $\S$\ref{composition}
in terms of Gauss-Dirichlet composition instead of
Bhargava's composition on $2\times2\times2$ cubes of integers.
Moreover, we shall construct
an isomorphism of the \textit{extended form class group}
$\mathcal{Q}_N(d_K)/\sim_\mathfrak{n}$
onto $\mathrm{Gal}(K_\mathfrak{n}/K)$ in terms
of Fricke invariants (Theorem \ref{explicitGalois}). To this end, it is necessary
to utilize the concept of
canonical bases for nontrivial integral ideals (Proposition \ref{Buellcanonical}). We hope that this result,
as some extension of Gauss's work, would enrich
the classical theory of complex multiplication for the
imaginary quadratic case.

\section {Extended form class groups as ray class groups}

Throughout this paper, we let $K$ be an imaginary quadratic field of discriminant $d_K$
with ring of integers $\mathcal{O}_K$. Furthermore, let $\mathfrak{n}$ be a nontrivial ideal of $\mathcal{O}_K$
and $N$ be the least positive integer in $\mathfrak{n}$.
In this section, by using an equivalence relation
induced from a congruence subgroup of $\mathrm{SL}_2(\mathbb{Z})$ of level $N$, we shall construct an extended form
class group isomorphic to the ray class group $\mathcal{C}(\mathfrak{n})$.
\par
First, the following lemma explains why
we have to use $\mathcal{Q}_N(d_K)$ instead
of the whole of $\mathcal{Q}(d_K)$.

\begin{lemma}\label{prime}
If $Q=ax^2+bxy+cy^2\in\mathcal{Q}(d_K)$, then the lattice $[\omega_Q,\,1]$ \textup{(}$=\mathbb{Z}\omega_Q+\mathbb{Z}$\textup{)} is a fractional ideal of $K$ with $\mathrm{N}_{K/\mathbb{Q}}([\omega_Q,\,1])=1/a$.
\end{lemma}
\begin{proof}
See \cite[Theorem 7.7 (i)]{Cox} and \cite[Lemma 2.3 (iii)]{E-K-S}.
\end{proof}

Define an equivalence relation $\sim_\mathfrak{n}$ on
the set $\mathcal{Q}_N(d_K)$ by
\begin{equation*}
Q\sim_\mathfrak{n}Q'
\quad\Longleftrightarrow\quad
[[\omega_Q,\,1]]=[[\omega_{Q'},\,1]]~\textrm{in}~
\mathcal{C}(\mathfrak{n}).
\end{equation*}
Note that
$Q\sim_\mathfrak{n}Q'$ implies
$Q\sim_\mathfrak{m}Q'$
for every nontrivial ideal $\mathfrak{m}$ of $\mathcal{O}_K$ dividing $\mathfrak{n}$, since
there is a canonical homomorphism $\mathcal{C}(\mathfrak{n})
\rightarrow\mathcal{C}(\mathfrak{m})$.

\begin{theorem}\label{canonicalisomorphism}
The set of equivalence classes $\mathcal{Q}_N(d_K)/\sim_\mathfrak{n}$ can be
regarded as a group isomorphic to $\mathcal{C}(\mathfrak{n})$.
\end{theorem}
\begin{proof}
By the definition of $\sim_\mathfrak{n}$, the mapping
$\phi_\mathfrak{n}:\mathcal{Q}_N(d_K)/\sim_\mathfrak{n}
\rightarrow\mathcal{C}(\mathfrak{n})$ stated in (\ref{canonical}) is a well-defined injection.
Let $\sim_{\Gamma_1(N)}$
be the equivalence relation on $\mathcal{Q}_N(d_K)$ given in
(\ref{QQQQ}), and let $\mathfrak{N}=N\mathcal{O}_K$.
For $Q,\,Q'\in\mathcal{Q}_N(d_K)$, we deduce that
\begin{eqnarray*}
Q\sim_{\Gamma_1(N)}Q'&\Longleftrightarrow&
[[\omega_Q,\,1]]=[[\omega_{Q'},\,1]]~\textrm{in}~\mathrm{Cl}(\mathfrak{N})\\
&&\hspace{2cm}\textrm{because the mapping $\phi_N$ described in (\ref{QGC}) is bijective}\\
&\Longleftrightarrow&
Q\sim_\mathfrak{N}Q'\quad\textrm{by the definition of $\sim_\mathfrak{N}$}.
\end{eqnarray*}
Thus $\sim_{\Gamma_1(N)}$
 is the same as
$\sim_\mathfrak{N}$, from which we obtain the following commutative diagram:
\begin{figure}[H]
\begin{equation*}
\xymatrixcolsep{5pc}\xymatrix{
\mathcal{Q}_N(d_K)/\sim_{\Gamma_1(N)} \ar[r]^-\sim_{~~~~~\phi_N}  \ar@{->>}[d]&
\mathcal{C}(\mathfrak{N})
\ar@{->>}[d]^{\textrm{ canonical}}\\
\mathcal{Q}_N(d_K)/\sim_\mathfrak{n}
\ar[r]_-{\phi_\mathfrak{n}} &
\mathcal{C}(\mathfrak{n})
}
\end{equation*}
\caption{A commutative diagram for surjectivity of $\phi_\mathfrak{n}$}\label{comm}
\end{figure}
\noindent
It then follows that $\phi_\mathfrak{n}$ is surjective,
and hence it is bijective. Therefore one can
consider $\mathcal{Q}_N(d_K)/\sim_\mathfrak{n}$ as
a group isomorphic to $\mathcal{C}(\mathfrak{n})$
via the bijection $\phi_\mathfrak{n}$, as desired.
\end{proof}

Let $Q_0=x^2+b_Kxy+c_Ky^2$ be the principal form in $C(d_K)$ and
$\tau_K=\omega_{Q_0}$, namely,
\begin{equation*}
Q_0=\left\{
\begin{array}{ll}\displaystyle
x^2-\frac{d_K}{4}y^2
& \textrm{if}~d_K\equiv0\Mod{4},\\
\displaystyle x^2+xy+\frac{1-d_K}{4}y^2 & \textrm{if}~d_K\equiv1\Mod{4}
\end{array}\right.
\end{equation*}
and
\begin{equation*}
\tau_K=\frac{-b_K+\sqrt{d_K}}{2}=\left\{
\begin{array}{ll}\displaystyle
\frac{\sqrt{d_K}}{2}
& \textrm{if}~d_K\equiv0\Mod{4},\\
\displaystyle \frac{-1+\sqrt{d_K}}{2} & \textrm{if}~d_K\equiv1\Mod{4}.
\end{array}\right.
\end{equation*}
Then we have $\mathcal{O}_K=[\tau_K,\,1]$ (\cite[(5.14)]{Cox}).

\begin{proposition}\label{Buellcanonical}
If $\mathfrak{a}$ is a nontrivial ideal of $\mathcal{O}_K$, then
there is a unique triple of integers $(a,\,b,\,c)$ satisfying
$\mathfrak{a}=[a\tau_K+b,\,c]$ and
\begin{equation}\label{abc}
0<a\leq c,\quad0\leq b<c,\quad
a\,|\,c,\quad a\,|\,b.
\end{equation}
Here, $c$ is the smallest positive integer belonging to $\mathfrak{a}$.
\end{proposition}
\begin{proof}
See \cite[Theorem 6.9]{Buell89}.
\end{proof}

\begin{remark}\label{canonicalconverse}
\begin{itemize}
\item[(i)] We call such a $\mathbb{Z}$-basis $\{a\tau_K+b,\,c\}$
the \textit{canonical basis} for $\mathfrak{a}$.
\item[(ii)] Conversely, if $(a,\,b,\,c)$ is a triple of integers
satisfying (\ref{abc}) and
$c\,|\,\mathrm{N}_{K/\mathbb{Q}}(a\tau_K+b)$, then
\begin{equation*}
\{(a\tau_K+b)\nu_1+c\nu_2~|~\nu_1,\,\nu_2\in\mathcal{O}_K\}
\end{equation*}
is a nontrivial ideal of $\mathcal{O}_K$ whose
canonical basis is $\{a\tau_K+b,\,c\}$
(\cite[Proposition 6.12 and Theorem 6.15]{Buell89}).
\end{itemize}
\end{remark}

Let $\{\xi_1,\,\xi_2\}$ be the canonical basis for $\mathfrak{n}$, that is,
\begin{equation}\label{xat}
\begin{bmatrix}\xi_1\\\xi_2\end{bmatrix}=
\begin{bmatrix}a_1&a_2\\0&N\end{bmatrix}
\begin{bmatrix}\tau_K\\1\end{bmatrix}
\end{equation}
for some integers $a_1$ and $a_2$ satisfying
\begin{equation*}
0<a_1\leq N,\quad
0\leq a_2<N,\quad
a_1\,|\,N,\quad
a_1\,|\,a_2
\end{equation*}
by Proposition \ref{Buellcanonical}.
\par
For an integer $a$ relatively prime to $N$, we denote by $\widetilde{a}$ any integer satisfying $a\widetilde{a}\equiv1\Mod{N}$.
Furthermore, for $\alpha=\begin{bmatrix}p&q\\r&s\end{bmatrix}
\in\mathrm{GL}_2(\mathbb{R})$ with $\det(\alpha)>0$ and
$\tau\in\mathbb{H}$, let
$j(\alpha,\,\tau)=r\tau+s$.

\begin{lemma}\label{lemmapqrs}
Let $Q=ax^2+bxy+cy^2,\,Q'\in\mathcal{Q}_N(d_K)$. Then,
$Q\sim_\mathfrak{n}Q'$ if and only if
there is a matrix $\alpha=\begin{bmatrix}p&q\\r&s\end{bmatrix}$ in $\mathrm{SL}_2(\mathbb{Z})$ such that $Q=Q'^{\,\alpha}$ and
\begin{equation}\label{rs}
r\equiv0\Mod{a_1}\quad\textrm{and}\quad
s\equiv1+\widetilde{a}\left(\frac{a_2}{a_1}-
\frac{b_K-b}{2}\right)r\Mod{N}.
\end{equation}
\end{lemma}
\begin{proof}
Assume that $Q\sim_\mathfrak{n}Q'$. Since
$Q\sim_{\mathcal{O}_K}Q'$, we induce from
the isomorphism given in (\ref{CC}) that
\begin{equation}\label{QQg}
Q=Q'^{\,\gamma}\quad\textrm{for some}~\gamma\in\mathrm{SL}_2(\mathbb{Z}).
\end{equation}
And, we obtain by the definition of $\sim_\mathfrak{n}$ that
\begin{eqnarray}
[\omega_Q,\,1]&=&\lambda[\omega_{Q'},\,1]\quad
\textrm{for some}~\lambda\in K^*~\textrm{such that}~
\lambda\equiv^*1\Mod{\mathfrak{n}}\label{wlw}\\
&=&\lambda[\gamma(\omega_Q),\,1]\quad\textrm{by (\ref{QQg})}\nonumber\\
&=&\frac{\lambda}{j(\gamma,\,\omega_Q)}[\omega_Q,\,1].\nonumber
\end{eqnarray}
Thus we derive
\begin{equation}\label{lzj}
\lambda=\zeta j(\gamma,\,\omega_Q)
\quad\textrm{for some}~\zeta\in\mathcal{O}_K^*
\end{equation}
and get by (\ref{wlw})
\begin{equation*}
[\omega_Q,\,1]=\zeta j(\gamma,\,\omega_Q)[\omega_{Q'},\,1].
\end{equation*}
We then attain
\begin{equation}\label{alpha}
\zeta j(\gamma,\,\omega_Q)\begin{bmatrix}
\omega_{Q'}\\1
\end{bmatrix}=\alpha\begin{bmatrix}\omega_Q\\1\end{bmatrix}
\quad\textrm{for some}~\alpha=\begin{bmatrix}p&q\\r&s\end{bmatrix}\in
\mathrm{SL}_2(\mathbb{Z}),
\end{equation}
from which
\begin{equation*}
\omega_{Q'}=
\frac{\zeta j(\gamma,\,\omega_Q)\omega_{Q'}}{\zeta j(\gamma,\,\omega_Q)}=
\alpha(\omega_Q)
\end{equation*}
and
\begin{equation*}
r\omega_Q+s=\zeta j(\gamma,\,\omega_Q)=\lambda\equiv^*1\Mod{\mathfrak{n}}
\end{equation*}
by (\ref{lzj}). Since $a\omega_Q\in\mathcal{O}_K$, we deduce
\begin{equation*}
r(a\omega_Q)+as-a\in\mathfrak{n}=[a_1\tau_K+a_2,\,N],
\end{equation*}
and hence
\begin{equation*}
r\left(\tau_K+\frac{b_K-b}{2}\right)+as-a=(a_1\tau_K+a_2)e
+Nf\quad\textrm{for some}~e,\,f\in\mathbb{Z}.
\end{equation*}
This yields
\begin{equation*}
(r-a_1e)\tau_K+\left(\frac{b_K-b}{2}\,r+as-a-a_2e-Nf\right)=0
\end{equation*}
and so
\begin{equation*}
r=a_1e\quad\textrm{and}\quad
as=a+a_2e+Nf-\frac{b_K-b}{2}\,r.
\end{equation*}
Therefore, we achieve $Q=Q'^{\,\alpha}$ and (\ref{rs}).
\par
Conversely, assume that there is
$\alpha=\begin{bmatrix}p&q\\r&s\end{bmatrix}\in\mathrm{SL}_2(\mathbb{Z})$
satisfying $Q=Q'^{\,\alpha}$ and (\ref{rs}). We see that
\begin{eqnarray*}
aj(\alpha,\,\omega_Q)-a&=&
r(a\omega_Q)+as-a\\
&=&a_1e(a\omega_Q)+\left(\frac{a_2}{a_1}-\frac{b_K-b}{2}\right)
a_1e+aNf\quad\textrm{for some}~e,\,f\in\mathbb{Z}
\quad\textrm{by (\ref{rs})}\\
&=&(a_1\tau_K+a_2)e+N(af)\\
&\in&\mathfrak{n}.
\end{eqnarray*}
This implies by the fact $\gcd(N,\,a)=1$ that $j(\alpha,\,\omega_Q)\equiv^*1\Mod{\mathfrak{n}}$.
Since
\begin{equation*}
[\omega_{Q'},\,1]=[\alpha(\omega_Q),\,1]=\frac{1}{j(\alpha,\,\omega_Q)}
[\omega_Q,\,1],
\end{equation*}
we get $[[\omega_{Q'},\,1]]=[[\omega_Q,\,1]]$ in
$\mathcal{C}(\mathfrak{n})$, and hence
$Q\sim_\mathfrak{n}Q'$.
\end{proof}

\begin{remark}
The congruences (\ref{rs}) can be rewritten as
\begin{equation*}
\begin{bmatrix}r&s\end{bmatrix}
\begin{bmatrix}1&(b_K-b)/2\\0&a\end{bmatrix}
\begin{bmatrix}N/a_1&-a_2/a_1\\0&1\end{bmatrix}
\equiv\begin{bmatrix}0&1\end{bmatrix}
\begin{bmatrix}1&(b_K-b)/2\\0&a\end{bmatrix}
\begin{bmatrix}N/a_1&-a_2/a_1\\0&1\end{bmatrix}
\Mod{NM_{1,\,2}(\mathbb{Z})}
\end{equation*}
where $M_{1,\,2}(\mathbb{Z})$ is the
$\mathbb{Z}$-module of $1\times2$ matrices over $\mathbb{Z}$.
\end{remark}

Let
\begin{equation*}
\Gamma_\mathfrak{n}=\left\{
\begin{bmatrix}c_1&c_2\\
c_3&c_4\end{bmatrix}\in\mathrm{SL}_2(\mathbb{Z})~|~
c_1\equiv1\Mod{N},\,
c_2\equiv0\Mod{\frac{N}{a_1}},\,
c_3\equiv0\Mod{a_1},\,
c_4\equiv1\Mod{N}\right\}
\end{equation*}
which is a congruence subgroup of level $N$.
Let
\begin{equation*}
S=\begin{bmatrix}0&-1\\1&0\end{bmatrix}\quad
\textrm{and}\quad
T=\begin{bmatrix}1&1\\0&1\end{bmatrix}.
\end{equation*}

\begin{lemma}\label{lemmamn}
Let $\alpha=\begin{bmatrix}p&q\\
r&s\end{bmatrix}\in\mathrm{SL}_2(\mathbb{Z})$ such that
\begin{equation}\label{r0s1}
r\equiv0\Mod{a_1}\quad\textrm{and}\quad
s\equiv1+kr\Mod{N}~\textrm{for some}~k\in\mathbb{Z}.
\end{equation}
Then there is a pair of integers $(m,\,n)$ so that
$T^n\alpha T^m\in\Gamma_\mathfrak{n}$.
\end{lemma}
\begin{proof}
If we set $m=-k$, then
\begin{equation*}
\alpha T^m=\begin{bmatrix}p&q\\r&s\end{bmatrix}
\begin{bmatrix}1&-k\\0&1\end{bmatrix}=
\begin{bmatrix}p&q'\\
r&s'\end{bmatrix}\quad
\textrm{with}~q'=-pk+q~\textrm{and}~
s'=-rk+s.
\end{equation*}
Observe by (\ref{r0s1}) that
\begin{equation}\label{s'1N}
s'\equiv1\Mod{N}.
\end{equation}
Furthermore, if we let $n$ be an integer satisfying
\begin{equation}\label{nsq}
ns'\equiv-q'\Mod{\frac{N}{a_1}},
\end{equation}
then we have
\begin{equation*}
T^n\alpha T^m=\begin{bmatrix}1&n\\
0&1\end{bmatrix}\begin{bmatrix}p&q'\\
r&s'\end{bmatrix}=
\begin{bmatrix}p+nr&q'+ns'\\
r&s'\end{bmatrix}
\end{equation*}
with
\begin{equation*}
q'+ns'\equiv0\Mod{\frac{N}{a_1}}
\quad\textrm{and}\quad
p+nr\equiv1\Mod{N}
\end{equation*}
by (\ref{r0s1}), (\ref{s'1N}), (\ref{nsq})
and the fact $\det(T^n\alpha T^m)=1$.
This proves the lemma.
\end{proof}

Now, we shall show that the equivalence relation $\sim_\mathfrak{n}$ on $\mathcal{Q}_N(d_K)$ is essentially induced from the congruence subgroup $\Gamma_\mathfrak{n}$.

\begin{proposition}\label{inducedfrom}
Let $Q=ax^2+bxy+cy^2,\,Q'\in\mathcal{Q}_N(d_K)$. Then,
$Q\sim_\mathfrak{n}Q'$ if and only if
\begin{eqnarray*}
Q=Q'^{\,T^u\gamma T^v}&&\textrm{for some}~\gamma=\begin{bmatrix}c_1&c_2\\
c_3&c_4\end{bmatrix}\in\Gamma_\mathfrak{n}~\textrm{and}~
(u,\,v)\in\mathbb{Z}^2~\textrm{satisfying}\\
&&c_3v+c_4\equiv1+\widetilde{a}\left(\frac{a_2}{a_1}-
\frac{b_K-b}{2}\right)c_3\Mod{N}.
\end{eqnarray*}
\end{proposition}
\begin{proof}
Assume that $Q\sim_\mathfrak{n}Q'$. By Lemma \ref{lemmapqrs}, there is
$\alpha=\begin{bmatrix}p&q\\r&s\end{bmatrix}\in\mathrm{SL}_2(\mathbb{Z})$ such that $Q=Q'^{\,\alpha}$ and
\begin{equation}\label{r0ands1}
r\equiv0\Mod{a_1}\quad\textrm{and}\quad
s\equiv1+\widetilde{a}\left(\frac{a_2}{a_1}-
\frac{b_K-b}{2}\right)r\Mod{N}.
\end{equation}
It then follows from Lemma \ref{lemmamn} that there is $(m,\,n)\in\mathbb{Z}^2$ with
$T^n\alpha T^m\in\Gamma_\mathfrak{n}$.
If we set $\gamma=T^n\alpha T^m$, $u=-n$ and $v=-m$, then
we establish
\begin{equation*}
Q=Q'^{\,\alpha}=Q'^{\,T^u\gamma T^v}.
\end{equation*}
And, if we let $\gamma=\begin{bmatrix}c_1&c_2\\
c_3&c_4\end{bmatrix}$, then we see that
\begin{equation*}
\alpha=\begin{bmatrix}
p&q\\r&s
\end{bmatrix}=T^u\gamma T^v=
\begin{bmatrix}
1&u\\0&1
\end{bmatrix}\begin{bmatrix}
c_1&c_2\\c_3&c_4
\end{bmatrix}
\begin{bmatrix}
1&v\\0&1
\end{bmatrix}=
\begin{bmatrix}
c_1+uc_3&(c_1+uc_3)v+c_2+uc_4\\
c_3&c_3v+c_4
\end{bmatrix}.
\end{equation*}
Thus we attain $r=c_3$, and so
\begin{equation*}
c_3v+c_4=s\equiv1+\widetilde{a}\left(\frac{a_2}{a_1}-\frac{b_K-b}{2}
\right)c_3\Mod{N}
\end{equation*}
by (\ref{r0ands1}).
\par
Conversely, assume that there are $\gamma=\begin{bmatrix}c_1&c_2\\c_3&c_4\end{bmatrix}\in\Gamma_\mathfrak{n}$
and $(u,\,v)\in\mathbb{Z}^2$ such that $Q=Q'^{\,T^u\gamma T^v}$ and
\begin{equation*}
c_3v+c_4\equiv1+\widetilde{a}\left(\frac{a_2}{a_1}-
\frac{b_K-b}{2}\right)c_3\Mod{N}.
\end{equation*}
Since $\gamma\in\Gamma_\mathfrak{n}$ and
\begin{equation*}
T^u\gamma T^v=\begin{bmatrix}
c_1+uc_3&(c_1+uc_3)v+c_2+uc_4\\
c_3&c_3v+c_4
\end{bmatrix},
\end{equation*}
we get $c_3\equiv0\Mod{a_1}$. Therefore, we derive by Lemma
\ref{lemmapqrs} that $Q\sim_\mathfrak{n}Q'$.
\end{proof}

\begin{remark}
In particular, if $\mathfrak{n}=N\mathcal{O}_K$, then we obtain
$a_1=N$, $a_2=0$ and so $\Gamma_\mathfrak{n}=\Gamma_1(N)$.
Since $T\in\Gamma_1(N)$, we conclude by Proposition \ref{inducedfrom} that
$\sim_\mathfrak{n}$ is the same as $\sim_{\Gamma_1(N)}$ given
in (\ref{QQQQ}), which recovers the previous results of \cite{E-K-S} and \cite{J-K-S18}.
\end{remark}

\section {The composition law}\label{composition}

Now, we shall modify
the classical Gauss-Dirichlet composition in order
to achieve the composition law on the extended form class group $\mathcal{Q}_N(d_K)/\sim_\mathfrak{n}$.
Although this composition law can be explained directly from
Figure \ref{comm} and
\cite[Remark 2.10]{E-K-S}
for the case where $\mathfrak{n}=N\mathcal{O}_K$, we would like to
include this section for completeness.
\par
Let $Q=ax^2+bxy+cy^2,\,Q'=a'x^2+b'xy+c'y^2\in\mathcal{Q}_N(d_K)$.
By Theorem \ref{canonicalisomorphism} we must have
\begin{equation*}
[Q][Q']~\textrm{in}~\mathcal{Q}_N(d_K)/\sim_\mathfrak{n}
=\phi_\mathfrak{n}^{-1}\left(
\phi_\mathfrak{n}([Q])\phi_\mathfrak{n}([Q'])
\right)=
\phi_\mathfrak{n}^{-1}\left(
[[\omega_Q,\,1]][[\omega_{Q'},\,1]]
\right).
\end{equation*}
Let
$\mathfrak{a}=[\omega_Q,\,1][\omega_{Q'},\,1]$.
We get by Lemma \ref{prime} that
\begin{equation*}
\mathfrak{a}^{-1}=\frac{1}{\mathrm{N}_{K/\mathbb{Q}}(\mathfrak{a})}\,
\overline{\mathfrak{a}}=aa'\overline{\mathfrak{a}}=
[-a\overline{\omega}_Q,\,a][-a'\overline{\omega}_{Q'},\,a'],
\end{equation*}
which shows that $\mathfrak{a}^{-1}$ is an integral ideal in
the ray class $\left(\phi_\mathfrak{n}([Q])
\phi_\mathfrak{n}([Q'])\right)^{-1}$.
\par
Now, one can take a matrix $\gamma$ in $\mathrm{SL}_2(\mathbb{Z})$ so that
the new quadratic form
\begin{equation}\label{Q''}
Q''=a''x^2+b''xy+c''y^2=Q'^{\,\gamma}
\end{equation}
satisfies
$\gcd(a,\,a'',\,(b+b'')/2)=1$ (\cite[Lemmas 2.3 and 2.25]{Cox}).
We then obtain
\begin{equation}\label{B}
[\omega_Q,\,1][\omega_{Q''},\,1]=
\left[\frac{-B+\sqrt{d_K}}{2aa''},\,1
\right]
\end{equation}
where $B$ is a unique integer modulo $2aa''$ satisfying
\begin{equation*}
B\equiv b\Mod{2a},\quad
B\equiv b''\Mod{2a''}\quad\textrm{and}\quad
B^2\equiv d_K\Mod{4aa''}
\end{equation*}
(\cite[Lemma 3.2 and (7.13)]{Cox}).
Furthermore, we know
by (\ref{Q''}) and (\ref{B}) that
\begin{equation}\label{another}
\mathfrak{a}=[\omega_Q,\,1][\gamma(\omega_{Q''}),\,1]
=\frac{1}{j(\gamma,\,\omega_{Q''})}
[\omega_Q,\,1][\omega_{Q''},\,1]=\frac{1}{j(\gamma,\,\omega_{Q''})}\left[
\frac{-B+\sqrt{d_K}}{2aa''},\,1\right].
\end{equation}
\par
Let $\nu_1,\,\nu_2\in K^*$ such that
\begin{equation*}
\mathfrak{a}=[\nu_1,\,\nu_2]\quad\textrm{and}\quad
\nu=\frac{\nu_1}{\nu_2}\in\mathbb{H}.
\end{equation*}
By (\ref{another}) we may just take
\begin{equation*}
\nu_1=
\frac{-B+\sqrt{d_K}}{2aa''j(\gamma,\,\omega_{Q''})},
\quad
\nu_2=\frac{1}{j(\gamma,\,\omega_{Q''})}\quad
\textrm{and}\quad\nu=
\frac{-B+\sqrt{d_K}}{2aa''}.
\end{equation*}
Since $\mathfrak{a}^{-1}$ is an integral ideal of $K$,
we see $1\in\mathfrak{a}$ and so
\begin{equation}\label{1=uv}
1=u\nu_1+v\nu_2\quad\textrm{for some}~u,\,v\in\mathbb{Z}.
\end{equation}
Here, one can readily check $\gcd(N,\,u,\,v)=1$
because $\mathfrak{a}^{-1}$ is relatively prime to $\mathfrak{n}$.
Take a matrix
$\sigma=\begin{bmatrix}\mathrm{*}&\mathrm{*}\\
u'&v'\end{bmatrix}$ in $\mathrm{SL}_2(\mathbb{Z})$ such that
\begin{equation}\label{sigmauv}
\sigma\equiv\begin{bmatrix}\mathrm{*}&\mathrm{*}\\
u&v\end{bmatrix}\Mod{NM_2(\mathbb{Z})},
\end{equation}
which is possible by the surjectivity of the reduction $\mathrm{SL}_2(\mathbb{Z})\rightarrow\mathrm{SL}_2(\mathbb{Z}/N\mathbb{Z})$
(\cite[Lemma 1.38]{Shimura}).
If we set $\widetilde{\omega}=\sigma(\nu)$, then we deduce that
\begin{equation}\label{o1a}
[\widetilde{\omega},\,1]=[\sigma(\nu),\,1]
=\frac{1}{u'\nu+v'}[\nu,\,1]
=\frac{1}{u'\nu_1+v'\nu_2}[\nu_1,\,\nu_2]
=\frac{1}{u'\nu_1+v'\nu_2}\mathfrak{a}.
\end{equation}
Observe by (\ref{1=uv}) and (\ref{sigmauv}) that
\begin{equation*}
u'\nu_1+v'\nu_2-1=u'\nu_1+v'\nu_2-(u\nu_1+v\nu_2)
=(u'-u)\nu_1+(v'-v)\nu_2\in N\mathfrak{a}
\subseteq\mathfrak{n}\mathfrak{a},
\end{equation*}
and hence
\begin{equation*}
u'\nu_1+v'\nu_2\equiv^*1\Mod{\mathfrak{n}}.
\end{equation*}
This implies by (\ref{o1a})
\begin{equation*}
[[\widetilde{\omega},\,1]]=[\mathfrak{a}]\quad\textrm{in}~\mathcal{C}(\mathfrak{n}).
\end{equation*}
\par
Finally, if we let $\widetilde{Q}$
be the quadratic form in $\mathcal{Q}_N(d_K)$ satisfying $\omega_{\widetilde{Q}}=\widetilde{\omega}=\sigma(\nu)$, then
we attain
\begin{equation*}
[Q][Q']=[\widetilde{Q}].
\end{equation*}
More precisely, we have
\begin{equation*}
[Q][Q']=
\left[\left(aa''x^2+Bxy+\frac{B^2-d_K}{4aa''}y^2\right)^{\sigma^{-1}}\right].
\end{equation*}

\section {Generation of ray class fields by Fricke invariants}

From now on, we further assume that
$\mathfrak{n}$ is a proper ideal of $\mathcal{O}_K$ and so
$N\geq2$. In this section, we shall show that $K_\mathfrak{n}$
is in fact a specialization over $K$ of a certain modular function field
by utilizing Fricke invariants.
\par
For a lattice $\Lambda$ in $\mathbb{C}$, let
\begin{equation}\label{g2g3}
g_2(\Lambda)=60\sum_{\lambda\in\Lambda-\{0\}}
\frac{1}{\lambda^4},\quad
g_3(\Lambda)=140\sum_{\lambda\in\Lambda-\{0\}}
\frac{1}{\lambda^6},\quad
\Delta(\Lambda)=g_2(\Lambda)^3-27g_3(\Lambda)^2
\end{equation}
and
\begin{equation}\label{defj}
j(\Lambda)=1728\frac{g_2(\Lambda)^3}{\Delta(\Lambda)}.
\end{equation}
And, let $\wp(z;\,\Lambda)$ be the \textit{Weierstrass $\wp$-function} relative
to $\Lambda$ given by
\begin{equation}\label{wp}
\wp(z;\,\Lambda)=\frac{1}{z^2}+\sum_{\lambda\in\Lambda-\{0\}}
\left\{\frac{1}{(z-\lambda)^2}-\frac{1}{\lambda^2}\right\}
\quad(z\in\mathbb{C}).
\end{equation}
For a fractional ideal $\mathfrak{a}$ of $K$, the \textit{Weber function}
$h:\mathbb{C}/\mathfrak{a}\rightarrow\mathbb{P}^1(\mathbb{C})$
is defined by
\begin{equation*}
h(z;\,\mathfrak{a})=\left\{\begin{array}{ll}
\displaystyle\frac{g_2(\mathfrak{a})^2}{\Delta(\mathfrak{a})}
\,\wp(z;\,\mathfrak{a})^2 & \textrm{if}~K=\mathbb{Q}(\sqrt{-1}),\vspace{0.1cm}\\
\displaystyle\frac{g_3(\mathfrak{a})}{\Delta(\mathfrak{a})}
\,\wp(z;\,\mathfrak{a})^3 & \textrm{if}~K=\mathbb{Q}(\sqrt{-3}),\vspace{0.1cm}\\
\displaystyle\frac{g_2(\mathfrak{a})g_3(\mathfrak{a})}{\Delta(\mathfrak{a})}
\,\wp(z;\,\mathfrak{a}) & \textrm{otherwise}.
\end{array}\right.
\end{equation*}
Then, it follows from (\ref{g2g3}) and (\ref{wp}) that
\begin{equation}\label{weight0}
h(\nu z;\,\nu\mathfrak{a})=h(z;\,\mathfrak{a})\quad
\textrm{for any}~\nu\in K^*.
\end{equation}
As a consequence of the main theorem of complex multiplication,
we get the following result due to Hasse (\cite{Hasse}).

\begin{proposition}\label{Hassework}
If $\kappa$ is a generator of the $\mathcal{O}_K$-module
$\mathfrak{n}^{-1}/\mathcal{O}_K$, then
\begin{equation*}
K_\mathfrak{n}=H_K\left(h(\kappa;\,\mathcal{O}_K)\right)=
K\left(j(\mathcal{O}_K),\,h(\kappa;\,\mathcal{O}_K)\right).
\end{equation*}
\end{proposition}
\begin{proof}
See \cite[Corollary to Theorem 7 in Chapter 10]{Lang87}.
\end{proof}

Let $M_{1,\,2}(\mathbb{Q})$ be the set of $1\times2$ matrices over $\mathbb{Q}$. For
$i=1,\,2,\,3$ and
$\mathbf{v}=\begin{bmatrix}v_1&v_2\end{bmatrix}\in M_{1,\,2}(\mathbb{Q})-M_{1,\,2}(\mathbb{Z})$, the $i$th \textit{Fricke function} $f^{(i)}_\mathbf{v}(\tau)$ is defined on $\mathbb{H}$ by
\begin{equation*}
f^{(i)}_\mathbf{v}(\tau)=\left\{\begin{array}{ll}
\displaystyle\frac{g_2(\Lambda_\tau)g_3(\Lambda_\tau)}{\Delta(\Lambda_\tau)}
\,\wp(v_1\tau+v_2;\,\Lambda_\tau) & \textrm{if}~i=1,\vspace{0.1cm}\\
\displaystyle\frac{g_2(\Lambda_\tau)^2}{\Delta(\Lambda_\tau)}
\,\wp(v_1\tau+v_2;\,\Lambda_\tau)^2 & \textrm{if}~i=2,\vspace{0.1cm}\\
\displaystyle\frac{g_3(\Lambda_\tau)}{\Delta(\Lambda_\tau)}\,
\wp(v_1\tau+v_2;\,\Lambda_\tau)^3 & \textrm{if}~i=3
\end{array}\right.
\end{equation*}
where $\Lambda_\tau=[\tau,\,1]$.
Note by the definition (\ref{defj}) that
\begin{equation}\label{fff}
f^{(2)}_\mathbf{v}(\tau)=\frac{1}{20736}\frac{f^{(1)}_\mathbf{v}(\tau)^2}
{j(\tau)-1728}
\quad\textrm{and}\quad
f^{(3)}_\mathbf{v}(\tau)=\frac{1}{373248}\frac{f^{(1)}_\mathbf{v}(\tau)^3}
{j(\tau)(j(\tau)-1728)}.
\end{equation}
Let $\mathcal{F}_1=\mathbb{Q}(j(\tau))$ and
\begin{equation*}
\mathcal{F}_N=\mathbb{Q}\left(j(\tau),\,
f^{(i)}_\mathbf{v}(\tau)~|~
i=1,\,2,\,3~\textrm{and}~\mathbf{v}\in\frac{1}{N}M_{1,\,2}(\mathbb{Z})
-M_{1,\,2}(\mathbb{Z})\right)
\end{equation*}
where $j(\tau)=j(\Lambda_\tau)$.
Then it is well known that
$\mathcal{F}_N$ coincides with the field of meromorphic modular functions of level $N$ whose Fourier coefficients
belong to the $N$th cyclotomic field
(\cite[Proposition 6.9 (i)]{Shimura}).

\begin{proposition}\label{functionGalois}
\begin{itemize}
\item[\textup{(i)}] The field
$\mathcal{F}_N$ is a Galois extension of $\mathcal{F}_1$ whose Galois group is isomorphic to
$\mathrm{GL}_2(\mathbb{Z}/N\mathbb{Z})/\{\pm I_2\}$ and satisfies
\begin{equation}
f^{(i)}_\mathbf{v}(\tau)^{\gamma}=f^{(i)}_{\mathbf{v}\gamma}(\tau)
\quad(i=1,\,2,\,3~\textrm{and}~\gamma\in\mathrm{GL}_2(\mathbb{Z}/N\mathbb{Z})/\{\pm I_2\}).
\end{equation}
\item[\textup{(ii)}] If $f(\tau)\in\mathcal{F}_N$ and $\gamma\in
\mathrm{SL}_2(\mathbb{Z}/N\mathbb{Z})/\{\pm I_2\}$, then
\begin{equation*}
f(\tau)^\gamma=f(\widetilde{\gamma}(\tau))
\end{equation*}
where $\widetilde{\gamma}$ is any element of $\mathrm{SL}_2(\mathbb{Z})$
which maps to $\gamma$ through the reduction $\mathrm{SL}_2(\mathbb{Z})\rightarrow
\mathrm{SL}_2(\mathbb{Z}/N\mathbb{Z})/\{\pm I_2\}$.
\end{itemize}
\end{proposition}
\begin{proof}
\begin{itemize}
\item[(i)] See \cite[Theorem 6.6]{Shimura}.
\item[(ii)] See \cite[(6.1.3)]{Shimura}.
\end{itemize}
\end{proof}

Now, consider the index set
\begin{equation*}
V_N=\left\{\mathbf{v}=\begin{bmatrix}v_1&v_2\end{bmatrix}\in M_{1,\,2}(\mathbb{Q})~|~\textrm{$N$ is the smallest positive integer so that}~N\mathbf{v}\in M_{1,\,2}(\mathbb{Z})\right\}.
\end{equation*}
Kubert and Lang first defined
in \cite[$\S$2.1]{K-L}
a \textit{Fricke family} of level $N$ to be
a family $\{h_\mathbf{v}(\tau)\}_{\mathbf{v}\in V_N}$
of functions in $\mathcal{F}_N$ satisfying
\begin{itemize}
\item[(F1)] $h_\mathbf{u}(\tau)=h_\mathbf{v}(\tau)$ if $\mathbf{u}
\equiv\pm\mathbf{v}\Mod{M_{1,\,2}(\mathbb{Z})}$,
\item[(F2)] $h_\mathbf{v}(\tau)^\gamma=h_\mathbf{v}\gamma$
for all $\gamma\in\mathrm{GL}_2(\mathbb{Z}/N\mathbb{Z})/\{\pm I_2\}\simeq\mathrm{Gal}(\mathcal{F}_N/\mathcal{F}_1)$.
\end{itemize}
For example, for each $i=1,\,2,\,3$,
the family $\{f^{(i)}_\mathbf{v}(\tau)\}_{\mathbf{v}\in V_N}$
is a Fricke family of level $N$ by Proposition \ref{functionGalois} (i).
If we let $\mathrm{Fr}(N)$ be the set of all
Fricke families of level $N$, then it is natural for us to regard
it as a field under the operations
\begin{eqnarray*}
\{h_\mathbf{v}(\tau)\}_\mathbf{v}+
\{k_\mathbf{v}(\tau)\}_\mathbf{v}&=&
\{(h_\mathbf{v}+k_\mathbf{v})(\tau)\}_\mathbf{v},\\
\{h_\mathbf{v}(\tau)\}_\mathbf{v}\cdot
\{k_\mathbf{v}(\tau)\}_\mathbf{v}&=&
\{(h_\mathbf{v}k_\mathbf{v})(\tau)\}_\mathbf{v}.
\end{eqnarray*}
And, let $\mathcal{F}^1_N(\mathbb{Q})$ be the field of meromorphic modular functions for the congruence subgroup
\begin{equation*}
\Gamma^1(N)=\left\{\gamma\in\mathrm{SL}_2(\mathbb{Z})~|~
\gamma\equiv\begin{bmatrix}1&0\\
\mathrm{*}&1\end{bmatrix}\Mod{NM_2(\mathbb{Z})}
\right\}
\end{equation*}
with rational Fourier coefficients. Then it was shown by Eum and Shin that
\begin{equation*}
\mathcal{F}^1_N(\mathbb{Q})=
\mathbb{Q}\left(j(\tau),\,f^{(i)}_{\left[
\begin{smallmatrix}1/N&0\end{smallmatrix}\right]}(\tau)~|~
i=1,\,2,\,3\right)
\end{equation*}
and
$\mathrm{Fr}(N)$
is indeed isomorphic to $\mathcal{F}^1_N(\mathbb{Q})$ through
the map
\begin{equation}\label{familyisomorphism}
\mathrm{Fr}(N)\stackrel{\sim}{\rightarrow}\mathcal{F}^1_N(\mathbb{Q}),
\quad
\{h_\mathbf{v}(\tau)\}_\mathbf{v}\mapsto
h_{\left[\begin{smallmatrix}1/N&0\end{smallmatrix}\right]}(\tau)
\end{equation}
(\cite[Theorem 4.3 and Proposition 6.1]{E-S} and (\ref{fff})).
\par
Let $\{h_\mathbf{v}(\tau)\}_\mathbf{v}\in\mathrm{Fr}(N)$ and
$C\in\mathcal{C}(\mathfrak{n})$. Take any integral
ideal $\mathfrak{c}$ in the class $C$, and let $\omega_1,\,
\omega_2\in K^*$ such that
\begin{equation*}
\mathfrak{n}\mathfrak{c}^{-1}=[\omega_1,\,\omega_2]\quad
\textrm{and}\quad \omega=\frac{\omega_1}{\omega_2}\in\mathbb{H}.
\end{equation*}
Since $N\in\mathfrak{n}\mathfrak{c}^{-1}$, we get
\begin{equation*}
N=r\omega_1+s\omega_2\quad\textrm{for some}~r,\,s\in\mathbb{Z}.
\end{equation*}
Now, we define the \textit{Fricke invariant} $h(C)$ by
\begin{equation*}
h(C)=h_{\left[\begin{smallmatrix}r/N&s/N\end{smallmatrix}\right]}
(\omega)
\end{equation*}
if it is finite.

\begin{proposition}\label{transformation}
With the notations as above, $h(C)$ depends only on the class $C$,
not on the choice of $\mathfrak{c}$, $\omega_1$ and $\omega_2$.
It belongs to $K_\mathfrak{n}$ and satisfies
\begin{equation*}
h(C)^{\sigma_\mathfrak{n}(C')}=h(CC')\quad(C'
\in\mathcal{C}(\mathfrak{n})),
\end{equation*}
where $\sigma_\mathfrak{n}$ is the
Artin reciprocity map for modulus $\mathfrak{n}$.
\end{proposition}
\begin{proof}
See \cite[Theorem 1.1 in Chapter 11]{K-L}.
\end{proof}

\begin{remark}\label{h(1)}
\begin{itemize}
\item[(i)] The Fricke invariant $h(C)$ is a
generalization of the Siegel-Ramachandra invariant
(\cite{Ramachandra} and \cite{Siegel}).
\item[(ii)] When $C_0$ is the identity class of $\mathcal{C}(\mathfrak{n})$, one can take $\mathfrak{c}=\mathcal{O}_K$ and so
\begin{equation*}
\mathfrak{n}\mathfrak{c}^{-1}=\mathfrak{n}=
[\xi_1,\,\xi_2]\quad\textrm{and}
\quad N=0\cdot\xi_1+1\cdot\xi_2.
\end{equation*}
Thus we have
\begin{eqnarray*}
h(C_0)&=&h_{\left[\begin{smallmatrix}
0&1/N
\end{smallmatrix}\right]}(\xi)
\quad\textrm{where}~\xi=\frac{\xi_1}{\xi_2}\\
&=&h_{\left[\begin{smallmatrix}
1/N&0
\end{smallmatrix}\right]}(\tau)^S|_{\tau=\xi}\quad\textrm{by
(F1) and (F2) with}~S=\begin{bmatrix}0&-1\\1&0\end{bmatrix}\\
&=&h_{\left[\begin{smallmatrix}
1/N&0
\end{smallmatrix}\right]}(S(\xi))\quad\textrm{by
Proposition \ref{functionGalois} (ii)}\\
&=&h_{\left[\begin{smallmatrix}
1/N&0
\end{smallmatrix}\right]}(-1/\xi).
\end{eqnarray*}
\item[(iii)] In particular, consider the Fricke family
$\{f^{(i)}_\mathbf{v}(\tau)\}_{\mathbf{v}\in V_N}$
for $i=1,\,2,\,3$ .
We then derive that
\begin{eqnarray*}
f^{(i)}(C_0)&=&f^{(i)}_{\left[\begin{smallmatrix}0&1/N\end{smallmatrix}\right]}
(\xi)\quad\textrm{by (ii)}\\
&=&\left\{\begin{array}{ll}
\displaystyle\frac{g_2([\xi_1,\,\xi_2])g_3([\xi_1,\,\xi_2])}{\Delta([\xi_1,\,\xi_2])}
\,\wp\left(\frac{\xi_2}{N};\,[\xi_1,\,\xi_2]\right) & \textrm{if}~i=1,\vspace{0.1cm}\\
\displaystyle\frac{g_2([\xi_1,\,\xi_2])^2}{\Delta([\xi_1,\,\xi_2])}
\,\wp\left(\frac{\xi_2}{N};\,[\xi_1,\,\xi_2]\right)^2 & \textrm{if}~i=2,\vspace{0.1cm}\\
\displaystyle\frac{g_3([\xi_1,\,\xi_2])}{\Delta([\xi_1,\,\xi_2])}\,
\wp\left(\frac{\xi_2}{N};\,[\xi_1,\,\xi_2]\right)^3 & \textrm{if}~i=3,
\end{array}\right.\quad\textrm{by (\ref{weight0})}\\
&=&h(1;\,\mathfrak{n})\quad\textrm{becase}~\xi_2=N.
\end{eqnarray*}
\end{itemize}
\end{remark}

\begin{proposition}\label{f(C)}
If $i=|\mathcal{O}_K^*|/2$, then
$K_\mathfrak{n}$ is generated by
$f^{(i)}(C)$ over $H_K$ for any $C\in\mathcal{C}(\mathfrak{n})$.
\end{proposition}
\begin{proof}
For the case where $K$ is different from $\mathbb{Q}(\sqrt{-1})$ and $\mathbb{Q}(\sqrt{-3})$,
see \cite[Theorem 3.2]{J-K-S16}.
\par
Now, let $K=\mathbb{Q}(\sqrt{-1})$ or $\mathbb{Q}(\sqrt{-3})$. Since the class number of $K$ is $1$,
we have $\mathfrak{n}=\nu\mathcal{O}_K$ for some
$\nu\in K^*$ and so $\mathfrak{n}^{-1}=\nu^{-1}\mathcal{O}_K$. Thus we obtain by Proposition \ref{Hassework}, (\ref{weight0}) and Remark \ref{h(1)} (iii) that
\begin{equation*}
K_\mathfrak{n}=H_K\left(h(\nu^{-1};\,\mathcal{O}_K)\right)
=H_K\left(h(1;\,\nu\mathcal{O}_K)\right)
=H_K\left(h(1;\,\mathfrak{n})\right)=H_K\left(f^{(i)}(C_0)\right).
\end{equation*}
Since $K_\mathfrak{n}$ is an abelian extension of $K$,
we further achieve
\begin{eqnarray*}
K_\mathfrak{n}&=&H_K\left(f^{(i)}(C_0)^{\sigma_\mathfrak{n}(C)}\right)
\\
&=&H_K\left(f^{(i)}(C_0C)\right)\quad
\textrm{by Proposition \ref{transformation}}\\
&=&H_K\left(f^{(i)}(C)\right).
\end{eqnarray*}
\end{proof}

\begin{corollary}\label{specialization}
We have
\begin{equation*}
K_\mathfrak{n}=
K\left(h(-1/\xi)~|~h(\tau)\in\mathcal{F}^1_N(\mathbb{Q})~\textrm{is finite at}~-1/\xi\right)\quad\textrm{with}~\xi=\frac{\xi_1}{\xi_2}=\frac{a_1\tau_K+a_2}{N}.
\end{equation*}
\end{corollary}
\begin{proof}
Since $\{j(\tau)\}_{\mathbf{v}\in V_N}$ is also
a Fricke family of level $N$ and $j(C_0)=j(\mathcal{O}_K)$, we get by Propositions
\ref{transformation} and \ref{f(C)} that
\begin{equation}\label{KKhC}
K_\mathfrak{n}=K\left(h(C_0)~|~\{h_\mathbf{v}(\tau)\}_\mathbf{v}\in
\mathrm{Fr}(N)~\textrm{such that $h(C_0)$ is finite}\right).
\end{equation}
We then deduce that
\begin{eqnarray*}
K_\mathfrak{n}
&=&K\left(h_{\left[\begin{smallmatrix}1/N&0\end{smallmatrix}\right]}
(-1/\xi)~|~
\{h_\mathbf{v}(\tau)\}_\mathbf{v}\in\mathrm{Fr}(N)~
\textrm{with}~h_{\left[\begin{smallmatrix}1/N&0\end{smallmatrix}\right]}(\tau)~\textrm{finite at}~-1/\xi\right)
\quad\textrm{by Remark \ref{h(1)} (ii)}\\
&=&K
\left(h(-1/\xi)~|~h(\tau)\in\mathcal{F}^1_N(\mathbb{Q})~\textrm{is finite at}~-1/\xi\right)\quad\textrm{by the isomorphism given in (\ref{familyisomorphism})}.
\end{eqnarray*}
\end{proof}

\section {Extended form class groups as Galois groups}

By the class field theory we know that
$\mathcal{C}(\mathfrak{n})$
is isomorphic to $\mathrm{Gal}(K_\mathfrak{n}/K)$
through the Artin map for modulus $\mathfrak{n}$.
Thus, the extended form class group $\mathcal{Q}_N(d_K)/\sim_\mathfrak{n}$
is also isomorphic to $\mathrm{Gal}(K_\mathfrak{n}/K)$.
In this section, we shall present
an explicit isomorphism of
$\mathcal{Q}_N(d_K)/\sim_\mathfrak{n}$ onto
$\mathrm{Gal}(K_\mathfrak{n}/K)$ in terms
of Fricke invariants.
\par
For each $Q=ax^2+bxy+cy^2\in\mathcal{Q}_N(d_K)$
there is a unique integer $d_Q$
satisfying
\begin{equation}\label{dQ}
d_Q\equiv-a_1\left(\frac{b+b_K}{2}\right)+a_2\Mod{N},\quad
d_Q\equiv0\Mod{a},\quad
0\leq d_Q+a_1\left(\frac{b+b_K}{2}\right)<Na
\end{equation}
by the fact $\gcd(N,\,a)=1$ and the Chinese remainder theorem.

\begin{lemma}\label{lemmadet}
Let $\mathfrak{a}$ be a nontrivial ideal of $\mathcal{O}_K$.
Let $\{\nu_1,\,\nu_2\}$ be a $\mathbb{Z}$-basis for $\mathfrak{a}$
such that $\nu_1/\nu_2\in\mathbb{H}$, and so
\begin{equation}\label{nAt}
\begin{bmatrix}\nu_1\\\nu_2\end{bmatrix}
=A\begin{bmatrix}\tau_K\\1\end{bmatrix}\quad
\textrm{for some}~A\in M_2(\mathbb{Z})~\textrm{such that}~\det(A)>0.
\end{equation}
Then we have $\det(A)=\mathrm{N}_{K/\mathbb{Q}}(\mathfrak{a})$.
\end{lemma}
\begin{proof}
We derive from (\ref{nAt})
\begin{equation*}
\begin{bmatrix}\nu_1&\overline{\nu}_1\\
\nu_2&\overline{\nu}_2\end{bmatrix}=
A\begin{bmatrix}\tau_K&\overline{\tau}_K\\
1&1\end{bmatrix}.
\end{equation*}
By taking determinant and squaring, we get
\begin{equation*}
\mathrm{disc}_{K/\mathbb{Q}}(\mathfrak{a})=\det(A)^2d_K.
\end{equation*}
Moreover, since $\mathrm{disc}_{K/\mathbb{Q}}(\mathfrak{a})=
\mathrm{N}_{K/\mathbb{Q}}(\mathfrak{a})^2d_K$ (\cite[Proposition 13 in Chapter III]{Lang94}) and $\det(A)>0$, we conclude
$\det(A)=\mathrm{N}_{K/\mathbb{Q}}(\mathfrak{a})$.
\end{proof}

\begin{lemma}\label{basislemma}
If $Q=ax^2+bxy+cy^2\in\mathcal{Q}_N(d_K)$ and
$\mathfrak{m}=[-a\overline{\omega}_Q,\,a]$, then we have $\mathfrak{n}\mathfrak{m}=[\xi_1',\,\xi_2']$ where
\begin{equation*}
\begin{bmatrix}\xi_1'\\\xi_2'\end{bmatrix}=
\begin{bmatrix}a_1&d_Q\\0&Na\end{bmatrix}\begin{bmatrix}-a\overline{\omega}_Q\\1\end{bmatrix}.
\end{equation*}
\end{lemma}
\begin{proof}
The smallest positive integer in $\mathfrak{m}=[-a\overline{\omega}_Q,\,a]$ is $a$. Note by
the fact $\gcd(N,\,a)=1$ and Lemma
\ref{prime} that the ideals $\mathfrak{n}$ and $\mathfrak{m}$ are
relatively prime. Hence the smallest positive integer in $\mathfrak{n}\mathfrak{m}=\mathfrak{n}\cap
\mathfrak{m}$ is
$\mathrm{lcm}(N,\,a)=Na$.
Now, if $\{\xi_1',\,\xi_2'\}$ is the canonical basis for $\mathfrak{n}\mathfrak{m}$, then
\begin{equation}\label{xBt}
\begin{bmatrix}
\xi_1'\\\xi_2'
\end{bmatrix}=B
\begin{bmatrix}\tau_K\\1\end{bmatrix}
\quad\textrm{with}~B=\begin{bmatrix}
b_1&b_2\\0&Na
\end{bmatrix}
\end{equation}
for some unique pair of integers $(b_1,\,b_2)$ satisfying
\begin{equation*}
0<b_1\leq Na,\quad 0\leq b_2<Na,\quad b_1\,|\,Na,\quad b_1\,|\,b_2.
\end{equation*}
Since
\begin{eqnarray*}
\det(B)&=&\mathrm{N}_{K/\mathbb{Q}}(\mathfrak{n}\mathfrak{m})
\quad\textrm{by (\ref{xBt}) and Lemma \ref{lemmadet}}\\
&=&\mathrm{N}_{K/\mathbb{Q}}(\mathfrak{n})
\mathrm{N}_{K/\mathbb{Q}}(\mathfrak{m})\\
&=&(a_1N)a\quad\textrm{by (\ref{xat}), Lemmas \ref{prime} and \ref{lemmadet}},
\end{eqnarray*}
we have $b_1=a_1$. On the other hand, we attain
\begin{equation}\label{widetildea2}
\mathfrak{n}=[a_1\tau_K+a_2,\,N]
=\left[a_1(-a\overline{\omega}_Q)-a_1\left(\frac{b+b_K}{2}\right)+a_2,\,N\right].
\end{equation}
Observe that
\begin{eqnarray*}
\mathfrak{n}\mathfrak{m}=\mathfrak{n}\cap\mathfrak{m}&\ni&
\xi_1'-(a_1(-a\overline{\omega}_Q)+d_Q)\quad
\textrm{because}~a_1(-a\overline{\omega}_Q)+d_Q\in\mathfrak{n}
\cap\mathfrak{m}~\textrm{by (\ref{dQ}) and (\ref{widetildea2})}\\
&=&a_1\tau_K+b_2-(a_1(-a\overline{\omega}_Q)+d_Q)
\quad\textrm{by (\ref{xBt}) and the fact $b_1=a_1$}\\
&=&a_1\left(-a\overline{\omega}_Q-\frac{b_K+b}{2}\right)+b_2
-(a_1(-a\overline{\omega}_Q)+d_Q)\\
&=&b_2-d_Q-a_1\left(\frac{b_K+b}{2}\right).
\end{eqnarray*}
Here, since
$0\leq b_2,\,d_Q+a_1(b+b_K)/2<Na$ and $Na$ is the smallest positive integer in $\mathfrak{n}\mathfrak{m}$, we must have $b_2=d_Q+a_1(b+b_K)/2$. Thus we obtain by (\ref{xBt}) that
\begin{equation*}
\begin{bmatrix}\xi_1'\\\xi_2'\end{bmatrix}
=\begin{bmatrix}a_1&d_Q+a_1(b_K+b)/2\\0&Na\end{bmatrix}
\begin{bmatrix}
\tau_K\\1
\end{bmatrix}=
\begin{bmatrix}a_1&d_Q\\0&Na\end{bmatrix}
\begin{bmatrix}
\tau_K+(b_K+b)/2\\1
\end{bmatrix}=
\begin{bmatrix}a_1&d_Q\\0&Na\end{bmatrix}
\begin{bmatrix}
-a\overline{\omega}_Q\\1
\end{bmatrix}.
\end{equation*}
\end{proof}

\begin{theorem}\label{explicitGalois}
We have an isomorphism
\begin{eqnarray*}
\mathcal{Q}_N(d_K)/\sim_\mathfrak{n}&\rightarrow&\mathrm{Gal}(K_\mathfrak{n}/K)\\
\mathrm{[}Q\mathrm{]}=
[ax^2+bxy+cy^2]&\mapsto&
\left(h(-1/\xi)
\mapsto
h^{\widetilde{a}S}
\left(\begin{bmatrix}a_1&d_Q/a\\0&N\end{bmatrix}(-\overline{\omega}_Q)\right)~|~
h(\tau)\in\mathcal{F}^1_N(\mathbb{Q})~
\textrm{is finite at}~-1/\xi\right).
\end{eqnarray*}
\end{theorem}
\begin{proof}
Let $Q=ax^2+bxy+cy^2\in\mathcal{Q}_N(d_K)$ and
$C=[[\omega_Q,\,1]]$ in $\mathcal{C}(\mathfrak{n})$. Since
$a^{\varphi(N)}\equiv1\Mod{N}$, where $\varphi$ is the Euler totient function,
one can take
$\mathfrak{c}=a^{\varphi(N)}[\omega_Q,\,1]$ as an integral ideal in $C$.
We get by Lemma \ref{prime} that
\begin{equation*}
\mathfrak{c}\overline{\mathfrak{c}}=\mathrm{N}_{K/\mathbb{Q}}(\mathfrak{c})
\mathcal{O}_K=a^{2\varphi(N)}\mathrm{N}_{K/\mathbb{Q}}([\omega_Q,\,1])
\mathcal{O}_K=a^{2\varphi(N)-1}\mathcal{O}_K,
\end{equation*}
and so
\begin{equation*}
\mathfrak{c}^{-1}=\overline{\mathfrak{c}}\,a^{-2\varphi(N)+1}\mathcal{O}_K=
a^{-\varphi(N)+1}[-\overline{\omega}_Q,\,1].
\end{equation*}
Thus we establish by Lemma \ref{basislemma} that
\begin{equation}\label{nc-1}
\mathfrak{n}\mathfrak{c}^{-1}=
a^{-\varphi(N)}[a_1\tau_K+a_2,\,N][-a\overline{\omega}_Q,\,a]=
a^{-\varphi(N)}[a_1(-a\overline{\omega}_Q)+d_Q,\,Na]
\end{equation}
and
\begin{equation}\label{comb}
N=0\cdot\{a^{-\varphi(N)}(a_1(-a\overline{\omega}_Q)+d_Q)\}+
a^{\varphi(N)-1}(a^{-\varphi(N)}Na).
\end{equation}
Now, let $h(\tau)$ be an element of $\mathcal{F}^1_N(\mathbb{Q})$
which is finite at $-1/\xi$, and
let
$\{h_\mathbf{v}(\tau)\}_\mathbf{v}$
be the Fricke family of level $N$ induced from $h(\tau)$ via
the isomorphism mentioned in (\ref{familyisomorphism}).
Then we achieve that
\begin{eqnarray*}
h(-1/\xi)^{\sigma_\mathfrak{n}(C)}&=&
h(C_0)^{\sigma_\mathfrak{n}(C)}\quad\textrm{by Remark \ref{h(1)} (ii)}\\
&=&h(C_0C)\quad\textrm{by Proposition \ref{transformation}}\\
&=&h(C)\\
&=&h_{\left[\begin{smallmatrix}
0&a^{\varphi(N)-1}/N
\end{smallmatrix}\right]}\left(
\frac{a_1(-a\overline{\omega}_Q)+d_Q}{Na}\right)\quad
\textrm{by (\ref{nc-1}) and (\ref{comb})}\\
&=&h_{\left[\begin{smallmatrix}0&\widetilde{a}/N\end{smallmatrix}\right]}
\left(\begin{bmatrix}a_1a&d_Q\\0&Na\end{bmatrix}(-\overline{\omega}_Q)\right)\quad
\textrm{by (F1)}\\
&=&h_{\left[\begin{smallmatrix}1/N&0\end{smallmatrix}\right]
\widetilde{a}S}
\left(\begin{bmatrix}a_1&d_Q/a\\0&N\end{bmatrix}(-\overline{\omega}_Q)\right)
\quad\textrm{since}~\begin{bmatrix}a_1a&d_Q\\0&Na\end{bmatrix}~
\textrm{and}~
\begin{bmatrix}a_1&d_Q/a\\0&N\end{bmatrix}\\
&&\hspace{4cm}\textrm{yield the same fractional linear transformation on $\mathbb{H}$}\\
&=&h_{\left[\begin{smallmatrix}1/N&0\end{smallmatrix}\right]}^{\widetilde{a}S}
\left(\begin{bmatrix}a_1&d_Q/a\\0&N\end{bmatrix}(-\overline{\omega}_Q)\right)
\quad\textrm{by Proposition \ref{functionGalois} (i)}\\
&=&h^{\widetilde{a}S}
\left(\begin{bmatrix}a_1&d_Q/a\\0&N\end{bmatrix}(-\overline{\omega}_Q)\right)
\quad\textrm{due to the isomorphism in (\ref{familyisomorphism})}.
\end{eqnarray*}
This, together with Corollary \ref{specialization}, completes the proof.
\end{proof}

\begin{remark}
In particular, if $\mathfrak{n}=N\mathcal{O}_K$, then we derive that
\begin{equation*}
h^{\widetilde{a}S}
\left(\begin{bmatrix}a_1&d_Q/a\\0&N\end{bmatrix}(-\overline{\omega}_Q)\right)
=h^{\widetilde{a}S}
\left(\begin{bmatrix}N&0\\0&N\end{bmatrix}(-\overline{\omega}_Q)\right)
=h^{\widetilde{a}S}
(-\overline{\omega}_Q).
\end{equation*}
\end{remark}

\section {Examples of extended form class groups}

In this last section, we shall explain how to find
elements of $\mathcal{Q}_N(d_K)/\sim_\mathfrak{n}$ and
give a couple of concrete examples.
\par
For each $Q=ax^2+bxy+cy^2\in\mathcal{Q}_N(d_K)$,
let
\begin{equation*}
V_Q=\left\{\begin{bmatrix}u&v\end{bmatrix}\in M_{1,\,2}(\mathbb{Z})~|~
(u\omega_Q+v)\mathcal{O}_K\in P_K(N\mathcal{O}_K)\right\}.
\end{equation*}
Then, it is straightforward to show
\begin{equation*}
V_Q=\left\{\begin{bmatrix}u&v\end{bmatrix}\in M_{1,\,2}(\mathbb{Z})~|~
\gcd(N,\,Q(v,\,-u))=1\right\}
\end{equation*}
(\cite[Lemma 4.1]{E-K-S}). Furthermore, let
\begin{eqnarray*}
U_K&=&\{(m,\,n)\in\mathbb{Z}^2~|~m\tau_K+n\in\mathcal{O}_K^*\}\\
&=&\left\{\begin{array}{ll}
\{\pm(0,\,1)\} & \textrm{if}~K\neq\mathbb{Q}(\sqrt{-1}),\,\mathbb{Q}(\sqrt{-3}),\\
\{\pm(0,\,1),\,\pm(1,\,0)\} & \textrm{if}~K=\mathbb{Q}(\sqrt{-1}),\\
\{\pm(0,\,1),\,\pm(1,\,0),\,\pm(1,\,1)\} & \textrm{if}~K=\mathbb{Q}(\sqrt{-3})
\end{array}\right.
\end{eqnarray*}
(\cite[Exercise 5.9]{Cox}).

\begin{lemma}\label{VQ}
Let $Q=ax^2+bxy+cy^2\in\mathcal{Q}_N(d_K)$ and
$\begin{bmatrix}r&s\end{bmatrix},\,
\begin{bmatrix}u&v\end{bmatrix}\in V_Q$. Then,
$[(r\omega_Q+s)\mathcal{O}_K]=
[(u\omega_Q+v)\mathcal{O}_K]$ in $\mathcal{C}(\mathfrak{n})$ if and only if
\begin{equation}\label{simQ}
\begin{array}{l}
\phantom{\equiv}\begin{bmatrix}r&s\end{bmatrix}
\begin{bmatrix}1&(b_K-b)/2\\0&a\end{bmatrix}
\begin{bmatrix}N/a_1&-a_2/a_1\\
0&1\end{bmatrix}\vspace{0.1cm}\\
\equiv
\begin{bmatrix}u&v\end{bmatrix}
\begin{bmatrix}
1&(b_K-b)/2\\0&a
\end{bmatrix}
\begin{bmatrix}
-mb_K+n&-mc_K\\m&n
\end{bmatrix}
\begin{bmatrix}N/a_1&-a_2/a_1\\
0&1\end{bmatrix}
\Mod{NM_{1,\,2}(\mathbb{Z})}\\
\phantom{\equiv}\textrm{for some}~(m,\,n)\in U_K.
\end{array}
\end{equation}
\end{lemma}
\begin{proof}
We deduce that
\begin{eqnarray*}
&&[(r\omega_Q+s)\mathcal{O}_K]=[(u\omega_Q+v)\mathcal{O}_K]\quad\textrm{in}~
\mathcal{C}(\mathfrak{n})\\
&\Longleftrightarrow&
\frac{r\omega_Q+s}{u\omega_Q+v}\mathcal{O}_K\in P_{K,\,1}(\mathfrak{n})\\
&\Longleftrightarrow&\frac{r\omega_Q+s}{u\omega_Q+v}\equiv^*\zeta\Mod{\mathfrak{n}}\quad
\textrm{for some}~\zeta\in\mathcal{O}_K^*\\
&\Longleftrightarrow&r(a\omega_Q)+as\equiv\zeta
\{u(a\omega_Q)+av\}\Mod{\mathfrak{n}}\quad\textrm{since}~
\gcd(N,\,a)=1~\textrm{and}~a\omega_Q\in\mathcal{O}_K\\
&\Longleftrightarrow&
r\left(\tau_K+\frac{b_K-b}{2}\right)+as-
(m\tau_K+n)\left\{
u\left(\tau_K+\frac{b_K-b}{2}\right)+av
\right\}\in\mathfrak{n}\quad\textrm{for some}~(m,\,n)\in U_K\\
&\Longleftrightarrow&
\frac{1}{a_1}\left(r+mu\frac{b_K+b}{2}-mav-nu
\right)\xi_1+\frac{1}{N}
\left\{\left(-\frac{a_2}{a_1}\right)
\left(r+mu\frac{b_K+b}{2}-mav-nu
\right)\right.\\
&&\left.
+r\frac{b_K-b}{2}+as+muc_K-nu\frac{b_K-b}{2}-nav
\right\}\xi_2\in\mathfrak{n}=[\xi_1,\,\xi_2]\\\\
&&\textrm{because}~\tau_K^2+b_K\tau_K+c_K=0~
\textrm{and}~\begin{bmatrix}\tau_K\\1\end{bmatrix}
=\begin{bmatrix}a_1&a_2\\0&N\end{bmatrix}^{-1}
\begin{bmatrix}\xi_1\\\xi_2\end{bmatrix}\\
&\Longleftrightarrow&
\left\{\begin{array}{ll}\displaystyle
r\equiv\left(-m\frac{b_K+b}{2}+n\right)u
+mav\Mod{a_1},\\
\displaystyle\left(-\frac{a_2}{a_1}+\frac{b_K-b}{2}\right)r+as\equiv
\left\{\frac{a_2}{a_1}\left(m\frac{b_K+b}{2}-n\right)-mc_K+n\frac{b_K-b}{2}\right\}u\\
\displaystyle\hspace{4.7cm}+\left(-ma\frac{a_2}{a_1}+na\right)v\Mod{N}.
\end{array}\right.
\end{eqnarray*}
This proves the lemma.
\end{proof}

Now, define
an equivalence relation $\sim_Q$ on the set $V_Q$ as follows:
Let $\begin{bmatrix}r&s\end{bmatrix},\,\begin{bmatrix}
u&v\end{bmatrix}\in V_Q$. Then,
$\begin{bmatrix}r&s\end{bmatrix}\sim_Q\begin{bmatrix}
u&v\end{bmatrix}$ if and only if they satisfy
the congruence relation stated in (\ref{simQ}).

\begin{proposition}\label{algorithm}
There is an algorithm to find all elements of the extended form class group $\mathcal{Q}_N(d_K)/\sim_\mathfrak{n}$.
\end{proposition}
\begin{proof}
Let $Q_1,\,Q_2,\,\ldots,\,Q_h$ be all
of the reduced forms in $\mathcal{Q}(d_K)$.
One can take $\alpha_i\in\mathrm{SL}_2(\mathbb{Z})$ so that
$Q_i'=Q_i^{\alpha_i}$ belongs to $\mathcal{Q}_N(d_K)$ ($i=1,\,2,\,
\ldots,\,h$)
(\cite[Lemmas 2.3 and 2.25]{Cox}). Observe by the isomorphism given in (\ref{CC}) that
\begin{equation}\label{Clw}
\mathcal{C}(\mathcal{O}_K)=\left\{[[\omega_{Q_1'},\,1]],\,
[[\omega_{Q_2'},\,1]],\,\ldots,\,
[[\omega_{Q_h'},\,1]]\right\}.
\end{equation}
On the other hand, since the canonical homomorphism $\mathcal{C}(N\mathcal{O}_K)
\rightarrow\mathcal{C}(\mathfrak{n})$ is surjective, we have
\begin{equation*}
P_K(\mathfrak{n})/P_{K,\,1}(\mathfrak{n})
\simeq
P_K(N\mathcal{O}_K)/
(P_K(N\mathcal{O}_K)\cap P_{K,\,1}(\mathfrak{n})),
\end{equation*}
from which it follows by Lemma \ref{VQ} that
\begin{equation}\label{PPw}
P_K(\mathfrak{n})/P_{K,\,1}(\mathfrak{n})
=
\left\{[(u\omega_{Q_i'}+v)\mathcal{O}_K]~|~
\left[\begin{bmatrix}u&v\end{bmatrix}\right]\in V_{Q_i'}/\sim_{Q_i'}\right\}
\quad\textrm{for each}~i=1,2,\,\ldots,\,h.
\end{equation}
Now that the canonical homomorphism
$\pi_\mathfrak{n}:\mathcal{C}(\mathfrak{n})\rightarrow
\mathcal{C}(\mathcal{O}_K)$
is a surjection with $\mathrm{Ker}(\pi_\mathfrak{n})=
P_K(\mathfrak{n})/P_{K,\,1}(\mathfrak{n})$, we see
by (\ref{Clw}) and (\ref{PPw}) that
\begin{eqnarray*}
\mathcal{C}(\mathfrak{n})&=&\left\{
\left[\frac{1}{u\omega_{Q_i'}+v}[\omega_{Q_i'},\,1]\right]~|~
i=1,\,2,\,\ldots,\,h~\textrm{and}~
\left[\begin{bmatrix}u&v\end{bmatrix}\right]\in V_{Q_i'}/\sim_{Q_i'}
\right\}\\
&=&\left\{
\left[\left[
\gamma_{i,\,\left[\left[\begin{smallmatrix}u&v\end{smallmatrix}\right]
\right]}(\omega_{Q_i'}),\,1
\right]\right]~|~
i=1,\,2,\,\ldots,\,h~\textrm{and}~
\left[\begin{bmatrix}u&v\end{bmatrix}\right]\in V_{Q_i'}/\sim_{Q_i'}
\right\}\\
&&\textrm{where}~\gamma_{i,\,\left[\left[\begin{smallmatrix}u&v\end{smallmatrix}\right]
\right]}~
\textrm{is an element of $\mathrm{SL}_2(\mathbb{Z})$ satisfying}~\gamma_{i,\,\left[\left[\begin{smallmatrix}u&v\end{smallmatrix}\right]
\right]}\equiv\begin{bmatrix}\mathrm{*}&\mathrm{*}\\
u&v\end{bmatrix}\Mod{NM_2(\mathbb{Z})}\\
&=&\left\{
\left[\left[
\omega_{Q_i'^{\,\gamma_{i,\,\left[\left[\begin{smallmatrix}u&v\end{smallmatrix}\right]
\right]}^{-1}}},\,1
\right]\right]~|~
i=1,\,2,\,\ldots,\,h~\textrm{and}~
\left[\begin{bmatrix}u&v\end{bmatrix}\right]\in V_{Q_i'}/\sim_{Q_i'}
\right\}.
\end{eqnarray*}
Therefore we attain
\begin{equation*}
\mathcal{Q}_N(d_K)/\sim_\mathfrak{n}=
\left\{
\left[
Q_i'^{\,\gamma_{i,\,\left[\left[\begin{smallmatrix}u&v\end{smallmatrix}\right]
\right]}^{-1}}
\right]~|~
i=1,\,2,\,\ldots,\,h~\textrm{and}~
\left[\begin{bmatrix}u&v\end{bmatrix}\right]\in V_{Q_i'}/\sim_{Q_i'}
\right\}.
\end{equation*}
\end{proof}

By using Remark \ref{canonicalconverse}, Lemmas \ref{lemmapqrs}, \ref{VQ} and Proposition \ref{algorithm}, we present the following examples.

\begin{example}
Let $K=\mathbb{Q}(\sqrt{-5})$ with $d_K=-20$, and let $\mathfrak{n}
=[2\tau_K+4,\,6]$ with $N=6$. Note that since $2$ is ramified in $K$
and $3$ splits in $K$, $\mathfrak{n}$ has prime ideal factorization
$\mathfrak{n}=\mathfrak{p}_2^2\mathfrak{p}_3$, where
$\mathfrak{p}_2=[\tau_K+1,\,2]$ is the prime ideal of $K$ lying above $2$
and $\mathfrak{p}_3=[\tau_K+2,\,3]$ is a prime ideal of $K$ lying above $3$.
We know that there are two reduced forms of discriminant $-20$
\begin{equation*}
Q_1=x^2+5y^2\quad\textrm{and}\quad
Q_2=2x^2+2xy+3y^2.
\end{equation*}
Take
\begin{equation*}
Q_1'=Q_1\quad\textrm{and}\quad Q_2'=Q_2^
{\left[\begin{smallmatrix}
1&-1\\1&0
\end{smallmatrix}\right]}=7x^2-6xy+2y^2.
\end{equation*}
One can then get that
\begin{equation*}
V_{Q_1'}/\sim_{Q_1'}=\left\{\left[\begin{bmatrix}0 & 1\end{bmatrix}\right],\,
\left[\begin{bmatrix}1&0\end{bmatrix}\right]\right\}
\quad\textrm{and}\quad
V_{Q_2'}/\sim_{Q_2'}=\left\{\left[\begin{bmatrix}0 & 1\end{bmatrix}\right],\,
\left[\begin{bmatrix}1&3\end{bmatrix}\right]\right\}
\end{equation*}
and
\begin{equation*}
\gamma_{1,\,\left[\left[\begin{smallmatrix}0 & 1\end{smallmatrix}\right]\right]}=
\begin{bmatrix}1&0\\0&1\end{bmatrix},\quad
\gamma_{1,\,\left[\left[\begin{smallmatrix}1 & 0\end{smallmatrix}\right]\right]}=
\begin{bmatrix}0&-1\\1&0\end{bmatrix},\quad
\gamma_{2,\,\left[\left[\begin{smallmatrix}0 & 1\end{smallmatrix}\right]\right]}=
\begin{bmatrix}1&0\\0&1\end{bmatrix},\quad
\gamma_{2,\,\left[\left[\begin{smallmatrix}1 & 3\end{smallmatrix}\right]\right]}=
\begin{bmatrix}1&2\\1&3\end{bmatrix}.
\end{equation*}
Thus we obtain
\begin{eqnarray*}
\mathcal{Q}_6(-20)/\sim_\mathfrak{n}&=&
\left\{
X_0=\left[Q_1'^{\,\gamma_{1,\,\left[\left[\begin{smallmatrix}0&1\end{smallmatrix}\right]
\right]}^{-1}}\right]=[x^2+5y^2],\,
X_2=\left[Q_1'^{\,\gamma_{1,\,\left[\left[\begin{smallmatrix}1&0\end{smallmatrix}\right]
\right]}^{-1}}\right]=[5x^2+y^2],\right.\\
&&~\left. X_1=\left[Q_2'^{\,\gamma_{2,\,\left[\left[\begin{smallmatrix}0&1\end{smallmatrix}\right]
\right]}^{-1}}\right]=[7x^2-6xy+2y^2],\,
X_3=\left[Q_2'^{\,\gamma_{2,\,\left[\left[\begin{smallmatrix}1&3\end{smallmatrix}\right]
\right]}^{-1}}\right]=[83x^2-118xy+42y^2]\right\}
\end{eqnarray*}
with the following group table:
\begin{table}[H]
\begin{center}
\begin{tabular}{l||llll}
 & $X_0$ & $X_1$ & $X_2$ & $X_3$ \\
 \hline\hline
$X_0$ & $X_0$ & $X_1$ & $X_2$ & $X_3$ \\
$X_1$ & $X_1$ &  $X_2$ & $X_3$ & $X_0$ \\
$X_2$ & $X_2$ & $X_3$ & $X_0$  & $X_1$ \\
$X_3$ & $X_3$ & $X_0$ & $X_1$  & $X_2$
\end{tabular}
\caption{Group table of $\mathcal{Q}_6(-20)/\sim_{[
2\tau_K+4,\,6]}$}
\end{center}
\end{table}
\end{example}

\begin{example}
Let $K=\mathbb{Q}(\sqrt{-23})$ with $d_K=-23$, and let
$\mathfrak{n}=[3\tau_K+9,\,12]$ with $N=12$. There are
three reduced forms
\begin{equation*}
Q_1=x^2+xy+6y^2,\quad
Q_2=2x^2-xy+3y^2,\quad
Q_3=2x^2+xy+3y^2.
\end{equation*}
If we take
\begin{equation*}
Q_1'=Q_1,\quad
Q_2'=Q_2^{\left[\begin{smallmatrix}
2&-1\\3&-1\end{smallmatrix}\right]}=29x^2-21xy+4y^2,\quad
Q_3'=Q_3^{\left[\begin{smallmatrix}
2&-1\\3&-1\end{smallmatrix}\right]}=41x^2-31xy+6y^2,
\end{equation*}
then we get
\begin{eqnarray*}
V_{Q_i'}/\sim_{Q_i'}&=&
\left\{
\left[\begin{bmatrix}
0 &1
\end{bmatrix}\right],\,
\left[\begin{bmatrix}
0 &5
\end{bmatrix}\right],\,
\left[\begin{bmatrix}
2 &1
\end{bmatrix}\right],\,
\left[\begin{bmatrix}
2 &7
\end{bmatrix}\right]
\right\}\quad(i=1,\,3),\\
V_{Q_2'}/\sim_{Q_2'}&=&
\left\{
\left[\begin{bmatrix}
0 &1
\end{bmatrix}\right],\,
\left[\begin{bmatrix}
0 &5
\end{bmatrix}\right],\,
\left[\begin{bmatrix}
2 &3
\end{bmatrix}\right],\,
\left[\begin{bmatrix}
2 &9
\end{bmatrix}\right]
\right\}
\end{eqnarray*}
and
\begin{eqnarray*}
\begin{array}{llll}
\gamma_{i,\,\left[\left[\begin{smallmatrix}0 & 1\end{smallmatrix}\right]\right]}=\begin{bmatrix}1&0\\0&1\end{bmatrix},
&
\gamma_{i,\,\left[\left[\begin{smallmatrix}0 & 5\end{smallmatrix}\right]\right]}=\begin{bmatrix}5&2\\12&5\end{bmatrix},
&
\gamma_{i,\,\left[\left[\begin{smallmatrix}2 & 1\end{smallmatrix}\right]\right]}=\begin{bmatrix}-1&-1\\2&1\end{bmatrix},
&
\gamma_{i,\,\left[\left[\begin{smallmatrix}2 & 7\end{smallmatrix}\right]\right]}=\begin{bmatrix}1&3\\2&7\end{bmatrix}
\quad(i=1,\,3),\vspace{0.1cm}\\
\gamma_{2,\,\left[\left[\begin{smallmatrix}0 & 1\end{smallmatrix}\right]\right]}=\begin{bmatrix}1&0\\0&1\end{bmatrix},
&
\gamma_{2,\,\left[\left[\begin{smallmatrix}0 & 5\end{smallmatrix}\right]\right]}=\begin{bmatrix}5&2\\12&5\end{bmatrix},
&
\gamma_{2,\,\left[\left[\begin{smallmatrix}2 & 3\end{smallmatrix}\right]\right]}=\begin{bmatrix}1&1\\2&3\end{bmatrix},
&
\gamma_{2,\,\left[\left[\begin{smallmatrix}2 & 9\end{smallmatrix}\right]\right]}=\begin{bmatrix}1&4\\2&9\end{bmatrix}.
\end{array}
\end{eqnarray*}
Hence, one can derive that
\begin{eqnarray*}
&&\mathcal{Q}_{12}(-23)/\sim_{[3\tau_K+9,\,12]}\\
&=&
\left\{Y_{0,\,0,\,0}=\left[Q_1'^{\,\gamma_{1,\,\left[\left[\begin{smallmatrix}0&1\end{smallmatrix}\right]
\right]}^{-1}}\right]=[x^2+xy+6y^2],\,
Y_{1,\,0,\,0}=\left[Q_1'^{\,\gamma_{1,\,\left[\left[\begin{smallmatrix}0&5\end{smallmatrix}\right]
\right]}^{-1}}\right]=[829x^2-691xy+144y^2],\right.\\
&&~~Y_{0,\,1,\,0}=\left[Q_1'^{\,\gamma_{1,\,\left[\left[\begin{smallmatrix}2&1\end{smallmatrix}\right]
\right]}^{-1}}\right]=[23x^2+23xy+6y^2],\,
Y_{1,\,1,\,0}=\left[Q_1'^{\,\gamma_{1,\,\left[\left[\begin{smallmatrix}2&7\end{smallmatrix}\right]
\right]}^{-1}}\right]=[59x^2-53xy+12y^2],\\
&&~~Y_{0,\,0,\,1}=\left[Q_2'^{\,\gamma_{2,\,\left[\left[\begin{smallmatrix}0&1\end{smallmatrix}\right]
\right]}^{-1}}\right]=[29x^2-21xy+4y^2],\,
Y_{1,\,0,\,1}=\left[Q_2'^{\,\gamma_{2,\,\left[\left[\begin{smallmatrix}0&5\end{smallmatrix}\right]
\right]}^{-1}}\right]=[2561x^2-2089xy+426y^2],\\
&&~~Y_{1,\,1,\,1}=\left[Q_2'^{\,\gamma_{2,\,\left[\left[\begin{smallmatrix}2&3\end{smallmatrix}\right]
\right]}^{-1}}\right]=[403x^2-295xy+54y^2],\,
Y_{0,\,1,\,1}=\left[Q_2'^{\,\gamma_{2,\,\left[\left[\begin{smallmatrix}2&9\end{smallmatrix}\right]
\right]}^{-1}}\right]=[2743x^2-2461xy+552y^2],\\
&&~~Y_{0,\,0,\,2}=\left[Q_3'^{\,\gamma_{3,\,\left[\left[\begin{smallmatrix}0&1\end{smallmatrix}\right]
\right]}^{-1}}\right]=[41x^2-31xy+6y^2],\,
Y_{1,\,0,\,2}=\left[Q_3'^{\,\gamma_{3,\,\left[\left[\begin{smallmatrix}0&5\end{smallmatrix}\right]
\right]}^{-1}}\right]=[3749x^2-3059xy+624y^2],\\
&&~\left.Y_{1,\,1,\,2}=\left[Q_3'^{\,\gamma_{3,\,\left[\left[\begin{smallmatrix}2&1\end{smallmatrix}\right]
\right]}^{-1}}\right]=[127x^2+199xy+78y^2],\,
Y_{0,\,1,\,2}=\left[Q_3'^{\,\gamma_{3,\,\left[\left[\begin{smallmatrix}2&7\end{smallmatrix}\right]
\right]}^{-1}}\right]=[2467x^2-2149xy+468y^2]\right\}
\end{eqnarray*}
which is isomorphic to $\mathbb{Z}_2\times\mathbb{Z}_2\times\mathbb{Z}_3$
via the mapping
\begin{equation*}
\mathcal{Q}_{12}(-23)/\sim_{[3\tau_K+9,\,12]}\rightarrow
\mathbb{Z}_2\times\mathbb{Z}_2\times\mathbb{Z}_3,\quad
Y_{a,\,b,\,c}\mapsto(a,\,b,\,c).
\end{equation*}
\end{example}

\bibliographystyle{amsplain}

\address{
Department of Mathematics Education\\
Dongguk University-Gyeongju\\
Gyeongju-si, Gyeongsangbuk-do 38066\\
Republic of Korea} {zandc@dongguk.ac.kr}
\address{Applied Algebra and Optimization Research Center\\
Sungkyunkwan University\\
Suwon-si, Gyeonggi-do 16419\\
Republic of Korea} {hoyunjung@skku.edu}
\address{
Department of Mathematical Sciences \\
KAIST \\
Daejeon 34141\\
Republic of Korea} {jkkoo@math.kaist.ac.kr}
\address{
Department of Mathematics\\
Hankuk University of Foreign Studies\\
Yongin-si, Gyeonggi-do 17035\\
Republic of Korea} {dhshin@hufs.ac.kr}

\end{document}